\documentclass[12pt,a4paper,leqno]{amsart}

\usepackage{anysize}
\usepackage[utf8]{inputenc}
\usepackage{amsmath}
\usepackage{amssymb}
\usepackage{array}

\marginsize{3cm}{3cm}{3.5cm}{3.5cm}

\newtheorem{theorem}{Theorem}[section]

\newtheorem{remark}[theorem]{Remark}

\newcommand{\hpg}[5]{{}_{#1}\mbox{\rm F}_{\!#2}\! \left(\left.{#3 \atop #4}\right| #5 \right)}

\begin{document}

\title{A family of Ramanujan-Orr formulas for $1/\pi$}
\author{Jesús Guillera} 
\address{Av. Ces\'{a}reo Alierta, 31 esc. iz. {\rm $4^o$}--A, Zaragoza (Spain)}
\email{jguillera@gmail.com}
\keywords{Hypergeometric series; Ramanujan-Orr type formulas; Complete elliptic integral; Legendre's relation; Zeilberger's algorithm.}
\subjclass[2010]{Primary 33E05, 33C20; Secondary 11F03, 33C75, 33F10.}

\maketitle

\begin{abstract}
We use a variant of Wan's method to prove two Ramanujan-Orr type formulas for $1/\pi$. This variant needs to know in advance the formulas for $1/\pi$ that we want to prove, but avoids the need of solving a system of equations.
\end{abstract}

\section{Introduction}
In \cite{Wan}, Wan gives a method to derive a formula for $1/\pi$ from each identity of type
\begin{equation}\label{G}
G(x)=\frac{1}{\pi^2} K(a(x))K(b(x)),
\end{equation}
where $G$ is analytic near the origin, and $K$ is the complete elliptic integral of the first kind. The method uses only transformations of $K$, see \cite{Vidunas} or \cite[Chap. 19-20]{Be3} for those of higher degree, and the following Legendre's identity:
\begin{equation}\label{legendre}
-K(\sqrt{r_0})K(\sqrt{1-r_0})+K(\sqrt{r_0})E(\sqrt{1-r_0})+ E(\sqrt{r_0})K(\sqrt{1-r_0})=\frac{\pi}{2},
\end{equation}
where $E$ is the complete elliptic integral of the second kind \cite{Bo}. The aim of this short paper is to show a variant of his method which we can use when we know in advance the formula for $1/\pi$ that we want to prove (or disprove). This variant  avoids the need of solving a system of equations, because we apply the method via translation \cite{Zu}. This is simpler and faster. For example, suppose that the formula
\begin{equation}\label{ope}
\left\{ p_0 + p_1  \left( x \frac{d}{dx} \right)+p_2 \left( x \frac{d}{dx} \right)^2 + p_3 \left( x \frac{d}{dx} \right)^3 \right\} G = \frac{1}{\pi},
\end{equation}
where $p_0$, $p_1$, $p_2$, $p_3$ are known constants, holds at $x=x_0$. The translation consist in applying the same operator to the right side of (\ref{G}). Then, if $a(x_0)$ and $b(x_0)$ are complementary arguments of $K$, the action of the operator provides a proof of the identity if we take into account the Legendre's identity. The following factorization: 
\begin{multline}\label{eq-4F3}
\sum_{n=0}^{\infty} \frac{\left(\frac{s}{2}\right)_n\left(\frac{1-s}{2}\right)_n\left(\frac{1+s}{2}\right)_n\left(1-\frac{s}{2}\right)_n}{\left(\frac12\right)_n(1)_n^3} y^n= \\
\sum_{n=0}^{\infty} \frac{(s)_n(1-s)_n}{(1)_n^2}\left( \frac{1-\sqrt{1-y}+\sqrt{-y}}{2} \right)^n \cdot 
\sum_{n=0}^{\infty} \frac{(s)_n(1-s)_n}{(1)_n^2}\left( \frac{1-\sqrt{1-y}-\sqrt{-y}}{2} \right)^n,
\end{multline}
(see \cite[eq. 7.5.1.9]{PrMaBr}, and \cite{chan-tanigawa} for a proof of it) is of type (\ref{G}) for the values $s=1/2$, $s=1/4$, $s=1/3$, $s=1/6$. Another Orr-type factorization was used in \cite[Sect. 4]{Wan}.

\section{Wan's method via translation}

In this section we prove two couples of Ramanujan-Orr type formulas applying Wan's method via translation. 

\subsection{Example 1}

Leting $s=1/4$ and making $y=-4x^2(x-1)^2/(2x-1)^2$ in (\ref{eq-4F3}), we have
\begin{multline}\label{eq-4F3-2}
\sum_{n=0}^{\infty} \frac{\left(\frac18\right)_n \left(\frac38\right)_n\left(\frac58 \right)_n\left(\frac78\right)_n}{\left(\frac12\right)_n(1)_n^3} \left( \frac{-4x^2(x-1)^2}{(2x-1)^2} \right)^n=  \\
\sum_{n=0}^{\infty} \frac{\left(\frac14\right)_n \left(\frac34\right)_n}{(1)_n^2}x^n \sum_{n=0}^{\infty} \frac{\left(\frac14\right)_n\left(\frac34\right)_n}{(1)_n^2}\left( \frac{x}{2x-1} \right)^n.
\end{multline}
Applying the quadratic transformation
\begin{equation}\label{cuad-cuartos-medios}
\sum_{n=0}^{\infty} \frac{\left(\frac14\right)_n \left(\frac34\right)_n}{(1)_n^2} x^n = 
\frac{1}{\sqrt{1 + \sqrt{x}}} 
\sum_{n=0}^{\infty}  \frac{ \left(\frac12\right)_n^2 }{(1)_n^2}  \left( \frac{ 2\sqrt{x} }{1+\sqrt{x} } \right)^n, \quad \sum_{n=0}^{\infty}  \frac{ \left(\frac12\right)_n^2 }{(1)_n^2} x^n = \frac{2}{\pi} K \left( \sqrt{x} \right),
\end{equation}
to the two sums in the right side of (\ref{eq-4F3-2}), we get
\begin{equation}\label{eq-4F3-3}
\sum_{n=0}^{\infty} \frac{\left(\frac18\right)_n \left(\frac38\right)_n\left(\frac58 \right)_n\left(\frac78\right)_n}{\left(\frac12\right)_n(1)_n^3} \left( \frac{-4x^2(x-1)^2}{(2x-1)^2} \right)^n=
f(x) f\left( \frac{x}{2x-1} \right),
\end{equation}
where
\[
f(x)=\frac{2}{\pi} \frac{K\left( \sqrt{\frac{2\sqrt{x}}{1+\sqrt{x}}}\right)}{\sqrt{1+\sqrt{x}}}.
\]
The arguments of the two elliptic integrals are complementary at 
\[
x_0=\frac{1}{49}+\frac{4}{49} \sqrt{3} \, i, \qquad
y_0 = \frac{-4x_0^2(x_0-1)^2}{(2x_0-1)^2} = \frac{192}{2401}, \qquad r_0 = \frac{2\sqrt{x_0}}{1+\sqrt{x_0}}=\frac12+\frac{\sqrt{3}}{6} \, i.
\]
This suggest that there exist formulas for $1/\pi$ with $y_0=192/2401$. Inspired by this, using the PSLQ algorithm, we have discovered the following ones:
\begin{equation}\label{eq-3}
\sum_{n=0}^{\infty} \frac{\left(\frac18\right)_n\left(\frac38\right)_n\left(\frac58\right)_n\left(\frac78\right)_n}{\left(\frac12\right)_n(1)_n^3} \left( \frac{192}{2401} \right)^n \frac{376n^2+216n+15}{2n+1}=\frac{98\sqrt{21}}{9\pi},
\end{equation}
and 
\begin{equation}\label{eq-4}
\sum_{n=0}^{\infty} \frac{\left(\frac18\right)_n\left(\frac38\right)_n\left(\frac58\right)_n\left(\frac78\right)_n}{\left(\frac12\right)_n(1)_n^3} \left( \frac{192}{2401} \right)^n (70688n^3-9216n^2+1428n+90)=\frac{294\sqrt{21}}{\pi}.
\end{equation}
Hence, we can use Wan's method \cite{Wan} to prove identity (\ref{eq-4}). However, due to the fact that we already know the formula that we want to prove (thanks to the PSLQ algorithm), we do not need constructing a system of equations. Instead, we will use the method via translation \cite{Zu}; see \cite{GuiZu} for another application of translation. 
\begin{theorem}
Formula (\ref{eq-4}) is true.
\end{theorem}
\begin{proof}
First, by applying the operator
\begin{equation}\label{ope-1}
\left. 90 + 1428  \left( y \frac{d}{dy} \right)-9216 \left( y \frac{d}{dy} \right)^2 +70688 \left( y \frac{d}{dy} \right)^3 \right|_{y=y_0},
\end{equation}
to the left side of (\ref{eq-4F3-3}), we obtain
\begin{equation}\label{series-orr}
\sum_{n=0}^{\infty} \frac{\left(\frac18\right)_n\left(\frac38\right)_n\left(\frac58\right)_n\left(\frac78\right)_n}{\left(\frac12\right)_n(1)_n^3} \left( \frac{192}{2401} \right)^n (70688n^3-9216n^2+1428n+90).
\end{equation}
Then, applying the same operator, but written as
\begin{equation}\label{ope-2}
\left. 90 + 1428  \left( \frac{y}{y'} \frac{d}{dx} \right)-9216 \left( \frac{y}{y'} \frac{d}{dx} \right)^2 +70688 \left( \frac{y}{y'} \frac{d}{dx} \right)^3 \right|_{x=x_0},
\end{equation}
to the right side of (\ref{eq-4F3-3}), we obtain 
\begin{equation}\label{result}
\frac{588\sqrt{21}}{\pi^2} \left(\frac{}{} \! \! -K(\sqrt{r_0})K(\sqrt{1-r_0})+K(\sqrt{r_0})E(\sqrt{1-r_0})+ E(\sqrt{r_0})K(\sqrt{1-r_0}) \right).
\end{equation}
This was done automatically using the following Maple code: 
\begin{verbatim}
proof:=proc()
local x0,y,w,f,g,pr1,pr,K;  K:=x->EllipticK(x);
y:=x->-4*x^2*(x-1)^2/(2*x-1)^2; w:=x->y(x)/diff(y(x),x); 
f:=x->2/Pi*K(sqrt(2*sqrt(x)/(1+sqrt(x))))/sqrt(1+sqrt(x));
x0:=1/49*(1+4*sqrt(3)*I); g:=x->f(x)*f(x/(2*x-1)); 
pr1:=subs(x=x0,90*g(x)+1428*w(x)*diff(g(x),x)
-9216*w(x)*diff(w(x)*diff(g(x),x),x)
+70688*w(x)*diff(w(x)*diff(w(x)*diff(g(x),x),x),x));
pr:=combine(simplify(combine(simplify(expand(
simplify(evalc(simplify(pr1))))),radical)),radical);    
return pr;
end:
\end{verbatim}
Copy and paste this code in a Maple session. Executing the procedure \texttt{proof();} gives the output (\ref{result}). Finally, taking into account the Legendre's identity (\ref{legendre}), formula (\ref{eq-4}) follows. 
\end{proof}

The following theorem implies that the formulas (\ref{eq-3}) and (\ref{eq-4}) are equivalent.

\begin{theorem}
The identity
\begin{align}
\sum_{n=0}^{\infty} B(n,s) z^n & \left[  \frac{s(s^2-1)(s-2)+4(1+2s-2s^2)n+8(1+s-s^2)n^2}{2n+1}  \right. \nonumber \\ \label{form-equiv} & \left. - \frac{8(1-z)}{z}n^3+12n^2 \right]=0,
\end{align}
where
\[
B(n,s)=\frac{\left(\frac{s}{2}\right)_n\left(\frac{1-s}{2}\right)_n\left(\frac{1+s}{2}\right)_n\left(1-\frac{s}{2}\right)_n}{\left(\frac12\right)_n(1)_n^3},
\]
holds.
\end{theorem}
\begin{proof}
We can prove it automatically by computer \cite{PWZ} writing in a Maple session
\begin{verbatim}
with(SumTools[Hypergeometric]):
print(op=Zeilberger(v(n,s),s,n,S)[1]):
w:=(n,s)->subs({p=n,q=s},Zeilberger(v(p,q),q,p,S)[2]):
print(w0=w(0,s));
\end{verbatim}
where \texttt{v(n,s)} is the function inside the sum of (\ref{form-equiv}), and \texttt{S} is the operator such that \texttt{S v(n,s) = v(n,s+1)}. The output \texttt{op=1} (independent of \texttt{S}), means that
\[ 
\texttt{1 v(n,s)=v(n,s)=w(n+1,s)-w(n,s)}.
\] 
Then, as $\texttt{w(0,s)=0}$, identity (\ref{form-equiv}) is due to telescoping cancelation when we sum for ${\texttt n \geq 0}$.
\end{proof}

\begin{remark} \rm
It is possible to check the proof of (\ref{eq-3}) and (\ref{eq-4}) in a completely automatic way because a computer can verify that:
\begin{enumerate}
\item Both sides of (\ref{eq-4F3-3}) satisfy the same differential equation.
\item The action of the operator (\ref{ope-1}) to the left side of (\ref{eq-4F3-3}) gives (\ref{series-orr}). This is trivial, even for human.
\item The action of the operator (\ref{ope-2}) to the right side of (\ref{eq-4F3-3}) gives (\ref{result}). See the procedure \texttt{proof()}.
\item Identity (\ref{form-equiv}) is true (see the Maple code above).
\end{enumerate}
Observe that these steps are elementary in the sense that they do not require the use of modular functions. Observe in addition that the value of $\tau$, such that
\[
\frac{K(\sqrt{1-r_0})}{K(\sqrt{r_0})}=-2 i \tau, \qquad q=e^{2\pi i \tau},
\]
corresponds to the value of $r_0$ of the above proof, is not a quadratic irrational (in fact it seems to be transcendent).
\end{remark}

\subsection{Example 2}
Letting $s=1/4$ and replacing $x$ with $x/(x-1)$ in  (\ref{eq-4F3-2}), we have
\begin{multline}\label{orr-ex-2}
\sum_{n=0}^{\infty} \frac{\left(\frac18\right)_n \left(\frac38\right)_n\left(\frac58 \right)_n\left(\frac78\right)_n}{\left(\frac12\right)_n(1)_n^3} \left( \frac{-4x^2}{(x^2-1)^2} \right)^n= \\ \sum_{n=0}^{\infty} \frac{\left(\frac14\right)_n\left(\frac34\right)_n}{(1)_n^2}\left( \frac{x}{x-1} \right)^n \cdot 
\sum_{n=0}^{\infty} \frac{\left(\frac14\right)_n\left(\frac34\right)_n}{(1)_n^2}\left( \frac{x}{x+1} \right)^n,
\end{multline}
Then, using the known transformation
\[
\sum_{n=0}^{\infty} \frac{(\frac14)_n(\frac34)_n}{(1)_n^2} \left(\frac{x}{x-1}\right)^n=\sqrt[4]{1-x} \sum_{n=0}^{\infty} \frac{(\frac12)_n^2}{(1)_n^2} \left(\frac12 -\frac12\sqrt{1-x} \right)^n, 
\]
and the quadratic transformations (\ref{cuad-cuartos-medios})
and
\[
K(x)=\frac{1}{1+x} K \left( \frac{2\sqrt{x}}{1+x} \right),
\]
we deduce that
\begin{multline}\label{transf-ex-2}
\sum_{n=0}^{\infty} \frac{\left(\frac18\right)_n \left(\frac38\right)_n\left(\frac58 \right)_n\left(\frac78\right)_n}{\left(\frac12\right)_n(1)_n^3} \left( \frac{-4x^2}{(x^2-1)^2} \right)^n= \\ \frac{4}{\pi^2} \frac{\sqrt[4]{1-x}}{(1+g(x))\sqrt{1+h(x)}}K\left( \frac{2\sqrt{g(x)}}{1+g(x)}\right)K\left( \sqrt{\frac{2h(x)}{1+h(x)}} \right),
\end{multline}
where 
\[
g(x)=\sqrt{\frac12-\frac12\sqrt{1-x}}, \qquad h(x)=\sqrt {\frac{x}{x+1}}.
\]
The arguments of the elliptic $K$ integrals in (\ref{transf-ex-2}) are complementary at
\[
x_0=\frac{85\sqrt{41}-529}{128}, \quad y_0=-\frac{2^{14}}{23^4}=-\left(\frac{128}{529}\right)^2, \quad r_0=\frac{2h(x_0)}{1+h(x_0)} = \frac{33-5\sqrt{41}}{2},
\]
where $y(x)=-4x^2/(x^2-1)^2$. This suggest that there exist formulas for $1/\pi$ with $y_0=-2^{14} \cdot 23^{-4}$. With the help of the PSLQ algorithm we have found 
\begin{equation}\label{for1-ex-2}
\sum_{n=0}^{\infty} \frac{\left(\frac18\right)_n \left(\frac38\right)_n\left(\frac58 \right)_n\left(\frac78\right)_n}{\left(\frac12\right)_n(1)_n^3} \left( \frac{-2^{14}}{23^4} \right)^n \frac{6970n^2+4037n+280}{529(2n+1)}=\frac{\sqrt{23}}{3\pi},
\end{equation}
and the contiguous formula
\begin{equation}\label{for2-ex-2}
\sum_{n=0}^{\infty} \frac{\left(\frac18\right)_n \left(\frac38\right)_n\left(\frac58 \right)_n\left(\frac78\right)_n}{\left(\frac12\right)_n(1)_n^3} \left(  \frac{-2^{14}}{23^4} \right)^n \frac{-296225n^3-24576n^2+53002n+4200}{3174}=\frac{\sqrt{23}}{\pi}.
\end{equation}

\begin{theorem}
Formulas (\ref{for1-ex-2}) and  (\ref{for2-ex-2}) are true.
\end{theorem}

\begin{proof}
First, we apply the operator
\begin{equation}\label{ex2-ope-1}
\left. 4200 + 53002  \left( y \frac{d}{dy} \right)-24576 \left( y \frac{d}{dy} \right)^2 -296225 \left( y \frac{d}{dy} \right)^3 \right|_{y=y_0},
\end{equation}
to the left side of  (\ref{transf-ex-2}). Then, we use the Maple code
\begin{verbatim}
proof:=proc()
local x0,y,w,g,h,f,pr1,pr,prr,K; K:=x->EllipticK(x);
y:=x->-4*x^2/(x^2-1)^2; w:=x->y(x)/diff(y(x),x); 
x0:=-529/128+85/128*sqrt(41);
g:=x->sqrt(1/2-1/2*sqrt(1-x)); h:=x->sqrt(x/(x+1));
f:=x->4/Pi^2*(1-x)^(1/4)*1/((1+g(x))*sqrt(1+h(x)))
*K(2*sqrt(g(x))/(1+g(x)))*K(sqrt(2*h(x)/(1+h(x))));
pr1:=subs(x=x0,4200*f(x)+53002*w(x)*diff(f(x),x)
-24576*w(x)*diff(w(x)*diff(f(x),x),x)
-296225*w(x)*diff(w(x)*diff(w(x)*diff(f(x),x),x),x));
pr:=simplify(combine(simplify(expand(simplify(evalc(
simplify(pr1))))),radical));
prr:=simplify(radnormal(combine(radnormal(pr),radical))); 
return prr;
end:
\end{verbatim}
to apply the same operator, but written as
\begin{equation}\label{ex2-ope-2}
\left. 4200 + 53002  \left( \frac{y}{y'} \frac{d}{dx} \right)-24576 \left( \frac{y}{y'} \frac{d}{dx} \right)^2 -296225 \left( \frac{y}{y'} \frac{d}{dx} \right)^3 \right|_{x=x_0},
\end{equation}
to the right side of (\ref{transf-ex-2}). Executing the procedure and using the Legendre's identity, we prove  (\ref{for2-ex-2}), and from it we deduce (\ref{for1-ex-2}) by using the identity (\ref{form-equiv}).
\end{proof}

\section{Ramanujan-Orr type formulas and the PSLQ algorithm}

By the PSLQ algorithm we have discovered the formulas
\begin{equation}\label{eq-1}
\sum_{n=0}^{\infty} \frac{\left(\frac18\right)_n\left(\frac38\right)_n\left(\frac58\right)_n\left(\frac78\right)_n}{\left(\frac12\right)_n(1)_n^3} \frac{1}{7^{4n}}\frac{1920n^2+1072n+55}{2n+1}=\frac{196\sqrt{7}}{3\pi},
\end{equation}
and 
\begin{equation}\label{eq-2}
\sum_{n=0}^{\infty} \frac{\left(\frac13 \right)_n \left( \frac23 \right)_n \left( \frac16 \right)_n \left( \frac56 \right)_n}{\left(\frac12 \right)_n(1)_n^3} \left( \frac35 \right)^{6n} \frac{133n^2+79n+6}{2n+1}=\frac{625}{32\pi}.
\end{equation}
Hence, we believe that it should be possible to prove (\ref{eq-1}) and (\ref{eq-2}) from (\ref{eq-4F3}), with $s=1/4$ and $s=1/3$ respectively. In addition we find curious that these formulas look similar to the following challenging series for $1/\pi^2$:
\begin{equation}\label{pi2-1920}
\sum_{n=0}^{\infty} \frac{\left(\frac12\right)_n\left(\frac18\right)_n\left(\frac38\right)_n\left(\frac58\right)_n\left(\frac78\right)_n}{(1)_n^5} \frac{1}{7^{4n}}(1920n^2+304n+15)=\frac{56\sqrt{7}}{\pi^2},
\end{equation}
and
\begin{equation}\label{pi2-532}
\sum_{n=0}^{\infty} \frac{ \left( \frac12 \right)_n \left( \frac13 \right)_n \left( \frac23 \right)_n \left( \frac16 \right)_n \left( \frac56 \right)_n}{(1)_n^5} \left( \frac35 \right)^{6n} (532n^2+126n+9)=\frac{375}{4\pi^2}.
\end{equation}
Formula (\ref{pi2-1920}) was conjectured in \cite{Gui}, and formula (\ref{pi2-532}) is joint with G. Almkvist, and was conjectured in \cite{AlGu}. The following example of formula for $1/\pi$:
\begin{equation}\label{eq-ten}
\sum_{n=0}^{\infty} \frac{\left(\frac{1}{10} \right)_n \left( \frac{3}{10} \right)_n \left( \frac{7}{10} \right)_n \left( \frac{9}{10} \right)_n}{\left(\frac12 \right)_n(1)_n^3} \frac{1}{2^{6n}} \frac{2100n^2+1160n+63}{2n+1} = \frac{200}{\pi},
\end{equation}
is such that there is no known Orr-type factorization from which we can derive it.
\par As an example of application of the PSLQ algorithm, we explain how we have discovered the formula (\ref{for1-ex-2}): Let
\[
t(j)=\sum_{n=0}^{\infty}\frac{\left(\frac18\right)_n\left(\frac38\right)_n\left(\frac58\right)_n\left(\frac78\right)_n}{\left(\frac12\right)_n(1)_n^3}\left( \frac{-2^{14}}{23^4} \right) \frac{n^j}{2n+1}.
\]
Then, looking for integer relations among the quantities
\[
\frac{1}{\pi^2}, \quad  t(0)^2, \quad t(1)^2, \quad t(2)^2, \quad t(0)t(1), \quad t(0)t(2), \quad t(1)t(2),
\]
we guess that
\begin{align*}
\frac{6436343}{\pi^2} = & 705600 \, t(0)^2 + 146676321 \, t(1)^2 + 437228100 \, t(2)^2 \\ & + 20346480 \, t(0)t(1) + 35128800 \, t(0)t(2) + 506482020 \, t(1)t(2),
\end{align*}
and taking the square root, we have
\[
\frac{529\sqrt{23}}{\pi}=3(6970 \, t(2) + 4037 \, t(1) + 280 \, t(0)).
\]
This is interesting because as we have explained before, when we guess a formula corresponding to complementary arguments of the elliptic integrals of an Orr-type factorization, then it is simple to prove.

\section*{Acknowledgement}

I thank Wadim Zudilin for encouraging me to extend these investigations with the use of modular equations in alternative bases.

\section{Addendum}

From (\ref{orr-ex-2}), we can arrive at
\begin{align}
\sum_{n=0}^{\infty} \frac{\left(\frac18\right)_n \left(\frac38\right)_n\left(\frac58 \right)_n\left(\frac78\right)_n}{\left(\frac12\right)_n(1)_n^3} \! \left( \! \frac{-4x^2}{(x^2-1)^2} \! \right)^n &= 
\frac{4}{\pi^2} \frac{\sqrt[4]{1-x}}{\sqrt{1+h(x)}} K \! \!  \left(\sqrt{g(x)}\right) K \! \! \left(\sqrt{\frac{2h(x)}{1+h(x)}} \right),
\nonumber \\ \label{last-orr} \text{where} \quad g(x) &= \frac12-\frac12 \sqrt{1-x}, \quad \text{and} \quad h(x)=\sqrt{\frac{x}{x+1}}. 
\end{align}
From (\ref{last-orr}), and in the same way used to derive the previous formulas in this paper, we can prove the following ``divergent" series for $1/\pi$:
\begin{align}
\sum_{n=0}^{\infty} \frac{\left(\frac18\right)_n \left(\frac38\right)_n\left(\frac58 \right)_n\left(\frac78\right)_n}{\left(\frac12\right)_n(1)_n^3} \left( \frac{-2^{14}}{7^4} \right)^n (-600-7518n-24576n^2-18785n^3) & \, ``\!=\!" \, \frac{98\sqrt{7}}{\pi}, \nonumber \\
\sum_{n=0}^{\infty} \frac{\left(\frac18\right)_n \left(\frac38\right)_n\left(\frac58 \right)_n\left(\frac78\right)_n}{\left(\frac12\right)_n(1)_n^3} \left( \frac{-2^{14}}{7^4}  \right)^n \frac{120+1273n+2210n^2}{2n+1} & \, `` \!=\!" \, \frac{49\sqrt{7}}{\pi}.
\end{align}
Since these last two formulas are ``divergent" their ``upside-down" are convergent, and inspired by \cite{GuiRog}, we have guessed that
\begin{align}
\sum_{n=1}^{\infty} \frac{\left(\frac12\right)_n(1)_n^3}{\left(\frac18\right)_n \left(\frac38\right)_n\left(\frac58 \right)_n\left(\frac78\right)_n} \left( \frac{-7^4}{2^{14}} \right)^n \frac{-600+7518n-24576n^2+18785n^3}{7^4 n^3} & \, {\overset{?} =} \, 2 \, {\rm L}_{-7}(2) \! - \! 1, \nonumber \\
\sum_{n=1}^{\infty} \frac{\left(\frac12\right)_n(1)_n^3}{\left(\frac18\right)_n \left(\frac38\right)_n\left(\frac58 \right)_n\left(\frac78\right)_n} \left( \frac{-7^4}{2^{14}}  \right)^n \frac{120-1273n+2210n^2}{7^4 (-2n+1) n^3} & \, {\overset{?} =} \, {\rm L}_{-7}(2).
\end{align}
A geometric interpretation (as a volume) of the arithmetical constant ${\rm L}_{-7}(2)$ is given in \cite{BBBZ}.

\end{document}